\documentclass[11pt,a4paper]{article}
\usepackage{etex} 
\usepackage[latin1]{inputenc}
\usepackage[T1]{fontenc}
\usepackage[right=2cm,left=2cm,top=2cm,bottom=2cm]{geometry}
\usepackage{ifthen}
\usepackage{enumerate}
\usepackage{enumitem}
\usepackage{amsmath}
\usepackage{amsthm}
\usepackage{amsfonts}
\usepackage{amssymb}
\usepackage{latexsym}
\usepackage{ae}
\usepackage{xcolor,colortbl} 
\usepackage{graphics,amsmath,amssymb}
\usepackage[np]{numprint}
\usepackage{hyperref}
\usepackage{multirow}
\usepackage{fourier} 

\npdecimalsign{\ensuremath{.}}

\newtheorem{theorem}{Theorem}[section]

\newtheorem{coro}[theorem]{Corollary}
\newtheorem{lem}[theorem]{Lemma}
\newtheorem{prop}[theorem]{Proposition}

\theoremstyle{definition}
\newtheorem{defin}[theorem]{Definition}
\newtheorem{exe}[theorem]{Example}

\theoremstyle{remark}
\newtheorem{remark}[theorem]{Remark}\numberwithin{equation}{section}

\makeatletter
\newcommand{\subjclass}[2][2020]{%
  \let\@oldtitle\@title%
  \gdef\@title{\@oldtitle\footnotetext{#1 \emph{Mathematics subject classification.} #2}}%
}
\newcommand{\keywords}[1]{%
  \let\@@oldtitle\@title%
  \gdef\@title{\@@oldtitle\footnotetext{\emph{Key words and phrases.} #1.}}%
}
\makeatother

\newcommand{\Z}{\mathbb {Z}}

\newcommand{\R}{\mathbb {R}}

\newcommand{\p}{\mathbb {P}}

\newcommand{\Addresses}{{
  \bigskip
  \footnotesize

  O.~Bordell\`es, \textsc{2 all\'{e}e de la combe, 43000 Aiguilhe, France.}\par\nopagebreak
  \textit{E-mail address}: \texttt{borde43@wanadoo.fr}

  \medskip

  L.~T\'oth, \textsc{Department of Mathematics, University of P\'ecs, Ifj\'us\'ag \'utja 6, 7624 P\'ecs, Hungary.}\par\nopagebreak
  \textit{E-mail address}: \texttt{ltoth@gamma.ttk.pte.hu}

}}

\DeclareMathOperator{\mx}{max}

\DeclareMathOperator{\Li}{Li}
\DeclareMathOperator{\Lib}{Li_\beta}
\DeclareMathOperator{\Lil}{Li_\ell}
\DeclareMathOperator{\lcm}{lcm}

\title{{Additive arithmetic functions meet the inclusion-exclusion principle, II}}
\date{}

\author{Olivier Bordell\`es and L\'aszl\'o T\'oth}



\begin{document}

\maketitle

\begin{abstract}
We unify in a large class of additive functions the results obtained in the first part of this work. The proof rests on series involving the Riemann zeta function and certain sums of primes which may have their own interest.
\end{abstract}

\section{Introduction}

\label{se:intro}

\subsection{The class $\mathcal{F} \left( r ,s,\ell; \lambda_1, \lambda_2\right)$}

In \cite{bortoth21}, asymptotic formulas for the sums
$$\sum_{n_1,\ldots,n_k\le x} f((n_1,\ldots,n_k)) \ \textrm{and} \ \sum_{n_1,\ldots,n_k\le x} f([n_1,\ldots,n_k])$$
where $\left( n_1,\ldots,n_k \right) $ and $\left[ n_1,\ldots,n_k \right]$ are respectively the gcd and the lcm of the integers $n_1, \dotsc, n_k$, have been derived for several additive functions $f$, such as
\begin{align*}
   & \Omega_\ell (n) := \sum_{p^\alpha \| n} \alpha^\ell \, ; \quad T_\ell (n) := \sum_{p^\alpha \| n} {\alpha + \ell - 1 \choose \ell} \, ; \quad A(n) := \sum_{p^\alpha \| n} \alpha p \, ; \\
   & A^\star(n) := \sum_{p^\alpha \| n} p \, ; \quad B(n) := A(n) - A^\star(n) \, ; \quad A_\ell(n) := \sum_{p^\alpha \| n} \alpha^\ell  p \, ;  
\end{align*}
where $\ell \in \Z_{\geqslant 0}$ is a fixed integer. Note that the functions $A$, $A^\star$ and $B$ were first introduced by Alladi and Erd\H{o}s in \cite{all77}. The idea of unifying all these results in a large class of additive function arises naturally and it is precisely the main aim of this paper. 

Alladi and  Erd\H{o}s proved in \cite{all77} that
$$\sum_{n \leqslant x} A(n) = \sum_{\substack{p^\alpha \leqslant x \\ \alpha \geqslant 1}} p \left \lfloor \frac{x}{p^\alpha} \right \rfloor = \frac{\zeta(2)}{2} \frac{x^2}{\log x} + O \left( \left( \frac{x}{\log x} \right)^2\right).$$
Incidentally, our work will extend and generalize this sum in Propositions~\ref{pro:sum_primes_1} and~\ref{pro:sum_primes_2} below. 

\medskip

We now define the class $\mathcal{F} \left( r ,s,\ell; \lambda_1, \lambda_2\right)$ we will deal with.

\begin{defin}
Let $r,s ,\ell \in \Z_{\geqslant 0}$ and $\lambda_1, \, \lambda_2 \in \R_{\geqslant 0}$. A real-valued additive arithmetic function $f$ belongs to the class $\mathcal{F} \left( r ,s,\ell; \lambda_1, \lambda_2\right)$ if:
\begin{enumerate}
   \item[\scriptsize $\triangleright$] For all primes $p$, $f(p) = \lambda_1 p^r$ ;
   \item[\scriptsize $\triangleright$] There exists a function $g_\ell : \Z_{\geqslant 2} \to \R$ such that, for all primes $p$ and all $\alpha \in \Z_{\geqslant 2}$, we have
   $$f \left( p^\alpha \right) - f \left( p^{\alpha-1} \right) = \lambda_2 p^s \times g_\ell (\alpha)$$
   with
   \begin{equation}
      \left| g_\ell (\alpha) \right| \leqslant C_0 \alpha^\ell \label{eq:g_l(alpha)}
   \end{equation}
   for some $C_0 \geqslant 0$ independent of $\alpha$, but may depend on $\ell$.
\end{enumerate}
\end{defin}

\begin{exe}
$\Omega_\ell ,T_\ell \in \mathcal{F} \left( 0,0,\ell-1 ; 1, 1 \right)$ where $\ell \in \Z_{\geqslant 0}$, $A \in \mathcal{F} \left( 1,1,0 ; 1, 1 \right)$, $A^\star \in \mathcal{F} \left( 1,0,0 ; 1, 0 \right)$, $A_\ell \in \mathcal{F} \left( 1,1,\ell-1 ; 1, 1 \right)$ where $\ell \in \Z_{\geqslant 0}$, and $B \in \mathcal{F} \left( 0,1,0 ; 0, 1 \right)$. By convention, $g_\ell = 0$ if $\ell < 0$.
\end{exe}

\subsection{Main results}

\begin{theorem}
\label{th:main_1}
Let  $f \in \mathcal{F} \left( r ,s,\ell; \lambda_1, \lambda_2\right)$ and $k \in \Z_{\geqslant 1}$. Assume $r \leqslant k-1$ and $s \leqslant 2k-1$. Then either
\begin{equation}
   \sum_{n_1,\dotsc,n_k \leqslant x} f \left( (n_1,\ldots,n_k) \right) = F_{r ,s,\ell}(k) \, x^k + O \left( R_k(x) \right) \tag{\textrm{Form} 1} \label{eq:form_1}   
\end{equation}
or, for all $N \in \Z_{\geqslant 1}$
\begin{equation}
   \sum_{n_1,\dotsc,n_k \leqslant x} f \left( (n_1,\ldots,n_k) \right) = C x^k \log \log x + G_{r ,s,\ell}(k) x^k + \frac{C x^k}{\log x} \sum_{h=0}^{N-1} \frac{A_{k,h}}{(\log x)^h}+ O \left( \frac{x^k}{(\log x)^{N+1}} \right) \tag{\textrm{Form} 2}\label{eq:form_2}   
\end{equation}
where $A_{k,h}$ is given in \eqref{eq:A_{k,h}}, according to the following cases.
   \vspace{0.25cm}
\begin{enumerate}
   \item[\scriptsize $\triangleright$] {\rm \textbf{Form 1}:} $r \leqslant k-2$ and $s \leqslant 2k-2$.
   \vspace{0.25cm}
   \begin{small}
   \begin{center}
   \begin{tabular}{|c|c|c|}
   \rowcolor{darkgray!20} $s$ & $F_{r ,s,\ell}(k)$  & $R_k(x)$ \\
   \hline
   & & \\
   $ \leqslant 2k-3$ & $\displaystyle \sum_p \left( \frac{\lambda_1}{p^{k-r}} + \lambda_2 \sum_{\alpha = 2}^\infty \frac{g_\ell(\alpha)}{p^{\alpha k-s}}\right)$ & $x^{k-1} (\log x)^{\ell+1}$ \\
   & & \\
   \hline
   & & \\
   $=2k-2$ & $\displaystyle \sum_p \left( \frac{\lambda_1}{p^{k-r}} + \lambda_2 \sum_{\alpha = 2}^\infty \frac{g_\ell(\alpha)}{p^{\alpha (k-2)+2}}\right)$ & $\dfrac{x^{k-1/2}}{\log x}$ \\   
   & & \\
   \hline
   \end{tabular}
   \end{center}
   \end{small} 
   \vspace{0.25cm}
   \item[\scriptsize $\triangleright$] {\rm \textbf{Form 2}:} $\left( r \leqslant k-2 \ \mathrm{and} \ s = 2k-1 \right) $ or $r = k-1$.
   \begin{small}
   \begin{center}
   \begin{tabular}{|c|c|c|c|}
   \rowcolor{darkgray!20} $r$ & $s$ & $C$ & $G_{r ,s,\ell}(k)$ \\
   \hline
   & & & \\
   $ \leqslant k-2$ & $=2k-1$ & $\lambda_2 g_\ell(2)$ & $\displaystyle \lambda_2(D_k - g_\ell(2) \log 2) + \lambda_1 \sum_p \frac{1}{p^{k-r}}$ \\
   & & & \\
   \hline
   & & & \\
   \multirow{2}{*}{$=k-1$} & $\leqslant 2k-2$ & $\lambda_1$ & $\displaystyle \lambda_1 M + \lambda_2 \sum_p \sum_{\alpha = 2}^\infty \frac{g_\ell(\alpha)}{p^{\alpha k-s}}$ \\ 
   & & & \\
   \cline{2-4}
   & & & \\
   & $= 2k-1$ & $\lambda_1+\lambda_2 g_\ell(2)$ & $\displaystyle \lambda_1 M + \lambda_2(D_k - g_\ell(2) \log 2)$ \\
   & & & \\
   \hline
   \end{tabular}
   \end{center}
   \end{small}
   \vspace{0.25cm}
   where $M$ is given in \eqref{eq:Mertens_cst} and $D_k$ is defined in \eqref{eq:D_k}.
\end{enumerate} 
\end{theorem}

\begin{remark}
The error term $x^{k-1} (\log x)^{\ell+1}$ in Form $1$ above may slightly be improved in
$$x^{k-1}\left( (\log \log x)^{\kappa_{r,k-2}} + (\log x)^{\ell} (\log \log x)^{\kappa_{s,2k-3} \times \kappa_{\ell,0}} \right)$$
as can be seen in Propositions~\ref{pro:sum_primes_1} and~\ref{pro:sum_primes_2} below, but this is almost irrelevant in practice.
\end{remark}

\begin{theorem}
\label{th:main_2}
Let  $f \in \mathcal{F} \left( r ,s,\ell; \lambda_1, \lambda_2\right)$ and $k \in \Z_{\geqslant 1}$. Assume $r \geqslant k$ or $s \geqslant 2k$, and define $\mu := \max \left( r, \frac{s-1}{2} \right)$. Then
$$
   \sum_{n_1,\dotsc,n_k \leqslant x} f \left( (n_1,\ldots,n_k) \right) = \frac{C x^{\mu+1}}{(\mu+1)\log x} \sum_{j=0}^{k-1} (-1)^{k-1-j} {k \choose j} \zeta(\mu+1-j)   + O \left( x^{\mu+1} \left( (\log x)^{\nu_k} + (\log x)^{-2} \right) \right)$$
   where
   \begin{small}
   \begin{center}
   \begin{tabular}{|c|c|c|}
   \rowcolor{darkgray!20} $s$ & $C$ & $\nu_k$ \\
   \hline
   & & \\
   $< 2r+1$ & $\lambda_1$ & $- \frac{r+1}{k}$ \\
   & & \\
   \hline
   & & \\
   $> 2r+1$ & $\lambda_2 g_\ell(2)$ & $1- \frac{s+1}{k}$ \\
   & & \\
   \hline
   & & \\
   $= 2r+1$ & $\lambda_1 + \lambda_2 g_\ell(2)$ & $- \frac{r+1}{k}$ \\
   & &  \\
   \hline
   \end{tabular}
   \end{center}
   \end{small}
\end{theorem}

\medskip

In \cite{seg66}, Segal proved that, if $f : \Z_{\geqslant 1} \to \R$ is a prime-independent additive function satisfying $f (p^\alpha) \ll 2^{\alpha/2}$, then
$$\sum_{n \leqslant x} f(n) = f(p) x \log \log x + A x + O \left( \frac{x}{\log x} \right)$$
where $A$ is a constant. Theorem~\ref{th:main_1} allows us to improve and generalize this result.

\begin{coro}
\label{cor:main_1}
Let  $f \in \mathcal{F} \left( 0 ,s,\ell; \lambda_1, \lambda_2\right)$ with $s \in \{0,1\}$ and $k \in \Z_{\geqslant 1}$.
\begin{enumerate}
   \item[\scriptsize $\triangleright$] If $k=1$, then for all $N \in \Z_{\geqslant 1}$
   $$\sum_{n \leqslant x} f(n) = C x \log\log x + G_{0,s,\ell}(1) \, x + \frac{C  x}{\log x} \sum_{h=0}^{N-1} \frac{A_{1,h}}{(\log x)^h} + O \left( \frac{x}{(\log x)^{N+1}}\right) $$
where $A_{1,h}$ is given in \eqref{eq:A_{k,h}}, and $C = \begin{cases} f(p), & \textrm{if\ } s = 0 \, ; \\ f(p)+\lambda_2 g_\ell(2) , & \textrm{if\ } s = 1. \end{cases}$
   \item[\scriptsize $\triangleright$] If $k \geqslant 2$, then
   $$\sum_{n_1,\dotsc,n_k \leqslant x} f \left( (n_1,\ldots,n_k) \right) = F_{0 ,s,\ell}(k) \, x^k + O \left( x^{k-1} (\log x)^{\ell+1} \right).$$
\end{enumerate}
\end{coro}

\medskip

For estimates with the lcm, we derive from the theorems above the following results.

\begin{theorem}
\label{th:main_3}
Let  $f \in \mathcal{F} \left( r ,s,\ell; \lambda_1, \lambda_2\right)$ and $k \in \Z_{\geqslant 1}$.
\begin{enumerate}
   \item[\scriptsize $\triangleright$] If $1 \leqslant r \leqslant k$ and $s < 2r+1$, then
$$\sum_{n_1,\ldots,n_k \leqslant x} f([n_1,\ldots,n_k]) = \frac{k \lambda_1 \zeta(r+1)}{r+1} \, \frac{x^{k+r}}{\log x} + O \left( \frac{x^{k+r}}{(\log x)^2}\right).$$
   \item[\scriptsize $\triangleright$] If $2 \leqslant s \leqslant 2k$ and $s > 2r+1$, then
$$\sum_{n_1,\ldots,n_k \leqslant x} f([n_1,\ldots,n_k]) = \frac{2k g_\ell(2) \lambda_2 \zeta \left (\frac{s+1}{2} \right )}{s+1} \, \frac{x^{k+\frac{s-1}{2}}}{\log x} + O \left( \frac{x^{k+\frac{s-1}{2}}}{(\log x)^2}\right).$$
\end{enumerate}
\end{theorem}

\medskip

The next result deals with the class of functions containing the prime-independent additive functions seen in Corollary~\ref{cor:main_1} above.

\begin{theorem}
\label{th:main_4}
Let  $f \in \mathcal{F} \left( 0 ,s,\ell; \lambda_1, \lambda_2\right)$ with $s \in \{0,1\}$ and $k \in \Z_{\geqslant 1}$. Then, for all $N \in \Z_{\geqslant 1}$
$$\sum_{n_1,\ldots,n_k \leqslant x} f([n_1,\ldots,n_k]) = k C x^k \log\log x + H_k x^k + \frac{kC  x^k}{\log x} \sum_{h=0}^{N-1} \frac{A_{1,h}}{(\log x)^h} + O \left( \frac{x^k}{(\log x)^{N+1}}\right) $$
where $A_{1,h}$ is given in \eqref{eq:A_{k,h}}, $C = \begin{cases} f(p), & \textrm{if\ } s = 0 \, ; \\ f(p)+\lambda_2 g_\ell(2) , & \textrm{if\ } s = 1 \, ; \end{cases}$ and
$$H_k := k \, G_{0,s,\ell}(1) + \sum_{j=2}^k (-1)^{j-1} {k \choose j} F_{0,s,\ell} (j).$$
\end{theorem}

\subsection{Examples}

\subsubsection{Improved results}

The theorems above enables us to get more precise estimates for certain functions studied in \cite{bortoth21}.

\begin{coro}
\label{cor:A_l}
Let $\ell \in \Z_{\geqslant 0}$. Then, for all $N \in \Z_{\geqslant 1}$
$$\sum_{n_1,n_2 \leqslant x} A_\ell \left( (n_1,n_2) \right)  = x^2 \log \log x + G_{1,1,\ell}(2) x^2 + \frac{x^2}{\log x} \sum_{h=0}^{N-1} \frac{A_{2,h}}{(\log x)^h} + O \left( \frac{x^2}{(\log x)^{N+1}}\right)$$
where
$$G_{1,1,\ell}(2) = \gamma + \sum_p \left( \log \left( 1 - \frac{1}{p} \right) + \left( 1 - \frac{1}{p^2}\right) \sum_{\alpha = 1}^\infty \frac{\alpha^\ell}{p^{2 \alpha - 1}}\right).$$ 
In particular, $G_{1,1,0}(2) = M$ and $G_{1,1,1}(2) = \gamma + \sum_p \left( \log \left( 1 - \frac{1}{p} \right) + \frac{p}{p^2-1}\right) \approx 0.4829$.
\end{coro}

\begin{coro}
\label{cor:B}
Let $k \in \Z_{\geqslant 1}$. Then, for all $N \in \Z_{\geqslant 1}$
$$\sum_{n \leqslant x} B(n) = x \log\log x + (D_1 - \log 2) x + \frac{x}{\log x} \sum_{h=0}^{N-1} \frac{A_{1,h}}{(\log x)^h} + O \left( \frac{x}{(\log x)^{N+1}}\right)$$
where $D_1 = \gamma + \sum_p \left( \log (1 - \frac{1}{p}) + \frac{1}{p-1} \right) \approx \np{1.034}$, and
$$\sum_{n_1,\ldots,n_k \leqslant x} B \left( \left[ n_1,\ldots,n_k \right] \right) = k x^k \log \log x + H_k x^k + \frac{k x^k}{\log x} \sum_{h=0}^{N-1} \frac{A_{1,h}}{(\log x)^h} + O \left( \frac{x^k}{(\log x)^{N+1}}\right)$$
where
$$H_k = k \left\lbrace \gamma + \sum_p \left( \log (1 - \frac{1}{p}) + \frac{p^{k-1}}{p^k-1} \right)  - \log 2 \right\rbrace + \sum_{j=2}^k (-1)^{j-1} {k \choose j} \sum_p \frac{1}{p^{j-1}(p^j-1)}.$$
\end{coro}

\subsubsection{Miscellaneous examples}

We first deal with the additive function $\omega_m$, where $m \geqslant 2$ is a fixed integer, defined by
$$\omega_m(n) := \sum_{p^m \mid n} 1.$$

\begin{coro}
\label{cor:omega_m}
Let $m \in \Z_{\geqslant 2}$ and $k \in \Z_{\geqslant 1}$. Then, for all $N \in \Z_{\geqslant 1}$,
$$\sum_{n_1,\ldots,n_k \leqslant x} \omega_m \left( (n_1,\ldots,n_k) \right)  = x^k \sum_p \frac{1}{p^{mk}} + O \left( x^{k-1} \log x + \frac{x}{(\log x)^N} \right)$$
and
$$\sum_{n_1,\ldots,n_k \leqslant x} \omega_m \left( [n_1,\ldots,n_k] \right)  = x^k \sum_{j=1}^k (-1)^{j-1} {k \choose j} \sum_p \frac{1}{p^{mj}} + O \left( \frac{x^k}{(\log x)^N} \right).$$
\end{coro}

\medskip

Our second example is picked up from the \textsc{oeis}. The sequence \texttt{A064372} deals with the additive function $f$ defined recursively by $f(1) = 1$ and $f \left( p^\alpha \right)  = f(\alpha)$ for all prime powers $p^\alpha$. Note that, if $n = p_1^{\alpha_1} \dotsb p_k^{\alpha_k}$ with $\alpha_j$ equal $1$ or are primes, then $f(n) = \omega(n)$. Furthermore, by strong induction, we may derive the inequality
\begin{equation}
   f(n) \leqslant \Omega(n) + 1 \quad \left( n \geqslant 1 \right). \label{eq:oeis}
\end{equation}
Indeed, this inequality clearly holds for $n=1$, and assuming its truth for $1,2,3,\dotsc,n-1$, and writing uniquely $n=p(n)^\alpha \times m$ where $p(n)$ is the smallest prime factor of $n$, $\alpha \in \left\lbrace 1,\dotsc,n-1\right\rbrace$ and $p(n) \nmid m$ with $m \in \left\lbrace 1,\dotsc,n-1\right\rbrace$, we get
$$f(n) = f(\alpha) + f(m) \leqslant \Omega(\alpha) + \Omega(m) + 2 \leqslant \alpha - 1 + \Omega(m) + 2 = \Omega(n) + 1.$$ 
Hence $f \in \mathcal{F}(0,0,1;1,1)$ by \eqref{eq:oeis}, and Corollary~\ref{cor:main_1} and Theorem~\ref{th:main_4} yield the following estimates.

\begin{coro}
\label{cor:f}
Let $k \in \Z_{\geqslant 1}$.
\begin{enumerate}
   \item[\scriptsize $\triangleright$] If $k=1$, then, for all $N \in \Z_{\geqslant 1}$
   $$\sum_{n \leqslant x} f(n) = x \log \log x + G_{0,0,1}(1) x + \frac{x}{\log x} \sum_{h=0}^{N-1} \frac{A_{1,h}}{(\log x)^h} + O \left( \frac{x}{(\log x)^{N+1}} \right)$$
   with $G_{0,0,1}(1) = \gamma + \sum_p \left( \log \left( 1 - \frac{1}{p} \right) + \left( 1 - \frac{1}{p} \right) \sum_{\alpha=1}^\infty \frac{f(\alpha)}{p^\alpha} \right)$.
   \item[\scriptsize $\triangleright$] If $k \geqslant 2$, then
   $$\sum_{n_1,\ldots,n_k \leqslant x} f \left( (n_1,\ldots,n_k) \right)  = x^k \sum_p \left( 1 - \frac{1}{p^k} \right) \sum_{\alpha=1}^\infty \frac{f(\alpha)}{p^{\alpha k}} + O \left( x^{k-1} (\log x)^2 \right).$$
   \item[\scriptsize $\triangleright$] If $k \geqslant 1$, then, for all $N \in \Z_{\geqslant 1}$
   $$\sum_{n_1,\ldots,n_k \leqslant x} f \left( [n_1,\ldots,n_k] \right)  = k x^k \log \log x + H_k x^k + \frac{k x^k}{\log x} \sum_{h=0}^{N-1} \frac{A_{1,h}}{(\log x)^h} + O \left( \frac{x^k}{(\log x)^{N+1}} \right)$$
   where
   \begin{multline*}
      H_k = k \left\lbrace \gamma + \sum_p \left( \log \left( 1 - \frac{1}{p} \right) + \left( 1 - \frac{1}{p} \right) \sum_{\alpha=1}^\infty \frac{f(\alpha)}{p^\alpha} \right) \right\rbrace \\
      + \sum_{j=2}^k (-1)^{j-1} {k \choose j} \sum_p \left( 1 - \frac{1}{p^j} \right) \sum_{\alpha=1}^\infty \frac{f(\alpha)}{p^{\alpha j}}.
   \end{multline*}
\end{enumerate}
\end{coro}

\subsection{Notation}

$\p$ is the set of primes, $\lfloor x \rfloor$ is the integer part of $x \in \R$, $\{x \} = x - \lfloor x \rfloor$ is its fractional part, $\left( n_1,\dotsc,n_k \right)$ and $\left[ n_1,\dotsc,n_k \right]$ are respectively the $\gcd$ and the $\lcm$ of the integers $n_1,\dotsc,n_k$. For all $i,j \in \Z_{\geqslant 0}$, $\kappa_{i,j}$ is the Kronecker symbol\footnote{This symbol is usually denoted by $\delta_{i,j}$, but we change here the notation to avoid confusion with the number-theoretic remainder function $\delta_c$.}  given by
$$\kappa_{i,j} := \begin{cases} 1, & \textrm{if\ } i = j \, ; \\ 0, & \textrm{otherwise}. \end{cases}$$
The function $x \mapsto \delta_c(x)$ will always refer to the usual number-theoretic remainder function defined by
$$\delta_c(x) := e^{-c(\log x)^{3/5} (\log \log x)^{-1/5}}$$
where $c > 0$ is absolute and does not need to be the same at each occurrence. The Mertens' constant is given by
\begin{equation}
   M := \gamma + \sum_p \left( \log \left( 1 - \frac{1}{p} \right) + \frac{1}{p} \right) \approx \np{0.2614972} \dotsc \label{eq:Mertens_cst}
\end{equation}
The function $g_\ell$ introduced above is always supposed to satisfy the condition \eqref{eq:g_l(alpha)}. It also should be mentioned that the estimates are given for sufficiently large $x > e$, as the growth may depend on some fixed parameters. For instance, the result of Proposition~\ref{pro:frac_parts} certainly holds for all $x \geqslant\left( \frac{3N \alpha}{e(\ell+1)}\right)^{\frac{6N \alpha}{\ell+1}}$ to ensure the necessary condition $x \geqslant (\log x)^{\frac{2N \alpha}{\ell+1}}$, and also $x \geqslant \left( \frac{36}{c^2} \right)^\alpha \left( N \times \frac{\alpha+\ell+1}{\ell+1} + 1 \right)^{2\alpha}$ which implies the inequality $\delta_c(x^{1/\alpha}) \leqslant (\log x)^{-N \times \frac{\alpha+\ell+1}{\ell+1} - 1}$.
 
\section{Tools}

\label{se:tools}

\subsection{Main tools from \cite{bortoth21}}

We state here without proof the main lemmas from\cite{bortoth21} we will use here.

\begin{lem}[Lemma~4.3 of \cite{bortoth21}]
\label{le:additiv_gcd}
Let $f$ be an additive function and let $k\in \Z_{\geqslant 1}$. Then 
$$\sum_{n_1,\ldots,n_k \leqslant x} f \left( (n_1,\ldots,n_k) \right) = \sum_{p \leqslant x} f(p) \left \lfloor \frac{x}{p} \right \rfloor^k + \sum_{\substack{p^{\alpha} \leqslant x \\ \alpha \geqslant 2}} \left(f(p^{\alpha}) - f(p^{\alpha-1})\right) \left \lfloor \frac{x}{p^{\alpha}} \right \rfloor^k.$$
\end{lem}

\begin{lem}[Proposition~2.1 of \cite{bortoth21}]
\label{le:key}
Let $f$ be an additive function and let $k\in \Z_{\geqslant 1}$. Then
$$\sum_{n_1,\ldots,n_k \leqslant x} f([n_1,\ldots,n_k]) =\sum_{j=1}^k (-1)^{j-1} {k \choose j} \left \lfloor x \right \rfloor^{k-j} \sum_{n_1,\ldots,n_j \leqslant x} f \left( (n_1,\ldots,n_j) \right).$$
\end{lem}

\begin{lem}[Lemma~4.5 of \cite{bortoth21}]
\label{le:toth}
Let $k \in \Z_{\geqslant 1}$, $\ell \in \Z_{\geqslant 0}$, $p$ be a prime and $z \geqslant 1$ be a real number satisfying $k \log p > \max \left( \frac{\ell \max(1,\log z)}{z} \, , \, \frac{\ell}{z} + \frac{1}{2} \right)$. Then
$$\sum_{\alpha > z} \frac{\alpha^\ell}{p^{\alpha k}} \leqslant  \frac{3 z^\ell}{p^{kz}}.$$
\end{lem}

\subsection{Transformations of certains sums}

\begin{lem}
\label{le:transformation_case_1}
Let $r \in \Z_{\geqslant 1}$. For all $k \in \Z_{\geqslant 1}$ satisfying $k \leqslant r$ and all $x > e$
$$\sum_{\frac{x}{\log x} < p \leqslant x} p^r \left \lfloor \frac{x}{p} \right \rfloor^k = \sum_{n_1 < \log x} \dotsb \sum_{n_k < \log x} \ \sum_{\frac{x}{\log x} < p \leqslant M_k x} p^r$$
where
\begin{equation}
   M_k := M_k \left (n_1, \dotsc, n_k \right) = \min \left( \frac{1}{n_1} , \dotsc, \frac{1}{n_k}\right). \label{eq:M_k}
\end{equation}
\end{lem}

\begin{proof}
Let $S_k(x)$ be the sum of the left-hand side. The result follows by taking $h=k-1$ in the identity
\begin{equation}
   S_k(x) = \sum_{n_{k-h} < \log x} \dotsb \sum_{n_k < \log x} \ \sum_{\frac{x}{\log x} < p \leqslant x \min \left( \frac{1}{n_{k-h}}, \dotsc, \frac{1}{n_k}\right) } p^r \left \lfloor \frac{x}{p} \right \rfloor^{k-h-1} \label{eq:h_r}
\end{equation}
which can be proved by induction on $h \in \{0,\dotsc,k-1\}$. Indeed
$$S_k(x) = \sum_{\frac{x}{\log x} < p \leqslant x} p^r \left \lfloor \frac{x}{p} \right \rfloor^{k-1} \sum_{n_k \leqslant x/p} 1 = \sum_{n_k < \log x} \ \sum_{\frac{x}{\log x} < p \leqslant \frac{x}{n_k}} p^r \left \lfloor \frac{x}{p} \right \rfloor^{k-1}$$
proving the case $h=0$, and assuming that the assertion \eqref{eq:h_r} holds for some $h \in \{0, \dotsc, k-1\}$, we derive
\begin{align*}
   S_k(x) &= \sum_{n_k < \log x} \dotsb \sum_{n_{k-h} < \log x} \ \sum_{\frac{x}{\log x} < p \leqslant x \min \left( \frac{1}{n_{k-h}}, \dotsc, \frac{1}{n_k}\right) } p^r \left \lfloor \frac{x}{p} \right \rfloor^{k-h-2} \sum_{n_{k-h-1} \leqslant x/p} 1 \\
   &= \sum_{n_k < \log x} \dotsb \sum_{n_{k-h} < \log x} \ \sum_{n_{k-h-1} < \log x} \ \sum_{\frac{x}{\log x} < p \leqslant x \min \left( \frac{1}{n_{k-h-1}}, \frac{1}{n_{k-h}}, \dotsc, \frac{1}{n_k} \right) } p^r \left \lfloor \frac{x}{p} \right \rfloor^{k-h-2}
\end{align*}
as required.
\end{proof}

\begin{lem}
\label{le:transformation_case_2}
Let $s \in \Z_{\geqslant 1}$ and $g_\ell : \Z_{\geqslant 2} \to \R$ be any function. For all $k \in \Z_{\geqslant 1}$ such that $s \geqslant 2k$ and all $x > e$
$$
   \sum_{\frac{\sqrt{x}}{\log x} < p \leqslant \sqrt{x}} p^s \sum_{2 \leqslant \alpha \leqslant \frac{\log x}{\log p}} g_\ell(\alpha) \left \lfloor \frac{x}{p^\alpha} \right \rfloor^k  = \sum_{\alpha = 2}^N g_\ell(\alpha)  \sum_{n_1 < x^{1-\alpha/2} (\log x)^\alpha} \dotsb \sum_{n_k < x^{1-\alpha/2} (\log x)^\alpha} \ \sum_{\frac{\sqrt{x}}{\log x} < p \leqslant (M_k x)^{1/\alpha}} p^s$$
where $M_k$ is given in \eqref{eq:M_k} and
\begin{equation}
   N := N(x) = \left \lfloor\dfrac{\log x}{\log (\sqrt{x}/\log x)} \right \rfloor. \label{eq:N}
\end{equation}
\end{lem}

\begin{proof}
The proof is similar to that of Lemma~\ref{le:transformation_case_1}. We leave the details to the reader.
\end{proof}

\subsection{Useful estimates}

In the following lemma, some specific notation are needed. For any $\beta \geqslant 0$, set
\begin{equation}
   \pi_\beta (x) := \sum_{p \leqslant x} p^\beta \quad \textrm{and} \quad \Lib (x) := \int_2^x \frac{t^\beta}{\log t} \, \textrm{d}t \quad \left( x \geqslant 2 \right). \label{eq:notation}
\end{equation}

\begin{lem}
\label{le:sum_p_beta}
Let $\beta \geqslant 0$ be fixed. Then there exists $c > 0$ such that, for all $x \geqslant e$ sufficiently large, we have
$$\pi_\beta (x) = \Lib (x) + O_\beta \left( x^{\beta + 1} \delta_c(x) \right).$$
\end{lem}

\begin{proof}
Follows at once from the Prime Number Theorem and partial summation.
\end{proof}

\begin{lem}
\label{le:useful_est}
Let $r,s \in \Z_{\geqslant 1}$ and $g_\ell$ be a function satisfying \eqref{eq:g_l(alpha)}. 
\begin{enumerate}
   \item[\scriptsize $\triangleright$] For all integers $k \geqslant 1$ such that $r \geqslant k$ and all $x > e$,
   $$\sum_{p \leqslant \frac{x}{\log x}} p^r \left \lfloor \frac{x}{p} \right \rfloor^k \ll \frac{x^{r+1}}{(\log x)^{r-k+2}}.$$
   \item[\scriptsize $\triangleright$] For all integers $k \geqslant 1$ such that $s \geqslant 2k$ and all $x > e$,
$$\sum_{p\leqslant \frac{\sqrt{x}}{\log x}} p^s \sum_{2 \leqslant \alpha \leqslant \frac{\log x}{\log p}} \left| g_\ell(\alpha) \right| \left \lfloor \frac{x}{p^\alpha} \right \rfloor^k \ll \frac{x^{(s+1)/2}}{(\log x)^{s+2-2k}}.$$
\end{enumerate}
\end{lem}

\begin{proof}
We have
$$\sum_{p \leqslant \frac{x}{\log x}} p^r \left \lfloor \frac{x}{p} \right \rfloor^k \leqslant x^k \sum_{p \leqslant \frac{x}{\log x}} p^{r-k} \leqslant x^k \left( \frac{x}{\log x} \right)^{r-k} \pi \left( \frac{x}{\log x} \right) \ll \frac{x^{r+1}}{(\log x)^{r-k+2}}$$
and the second sum is similar.
\end{proof}

\begin{lem}
\label{le:re_alpha_3}
Let $s \in \Z_{\geqslant 1}$ and $g_\ell$ be a function satisfying \eqref{eq:g_l(alpha)}. For all $k \in \Z_{\geqslant 1}$ such that $s \geqslant 2k$ and all $x > e$
$$\sum_{\alpha = 3}^N \left| g_\ell(\alpha) \right| \sum_{n_1 < x^{1-\alpha/2} (\log x)^\alpha} \dotsb \sum_{n_k < x^{1-\alpha/2} (\log x)^\alpha} \ \sum_{p \leqslant (M_k x)^{1/\alpha}} p^s \ll \frac{x^{(s+1)/2}}{(\log x)^{s+2-2k}}$$
where $M_k$ is given in \eqref{eq:M_k} and $N$ is defined in \eqref{eq:N}.
\end{lem}

\begin{proof}
Let $X := X(x,\alpha) = x^{1-\alpha/2} (\log x)^\alpha$. The sum does not exceed
$$\ll \sum_{\alpha = 3}^N \alpha^\ell \sum_{n_1 < X} \dotsb \sum_{n_k < X} \frac{(M_k x)^{(s+1)/\alpha}}{ \log M_k x} \ll \frac{N^\ell}{\log x} \sum_{\alpha = 3}^N x^{(s+1)/\alpha} \sum_{n_1 < X} \dotsb \sum_{n_k < X} M_k^{(s+1)/\alpha}$$
and the inequality $M_k \leqslant (n_1 \dotsb n_k)^{-1/k}$ and H\"{o}lder's inequality imply that the sum is
\begin{align*}
   & \ll \frac{N^\ell}{\log x} \sum_{\alpha = 3}^N x^{(s+1)/\alpha} \left( \sum_{n < X} \frac{1}{n^{\frac{s+1}{k \alpha}}} \right)^k \ll \frac{N^\ell}{\log x} \sum_{\alpha = 3}^N x^{(s+1)/\alpha} X^{k-1} \sum_{n < X} \frac{1}{n^{\frac{s+1}{\alpha}}} \\
   & \ll \frac{N^\ell}{\log x} \sum_{\alpha = 3}^N x^{(s+1)/\alpha} X^{k-1} \left( \int_1^X \frac{\textrm{d}t}{t^{(s+1)/\alpha}} + 1 \right) \\
   & \ll \frac{N^\ell}{\log x} \sum_{\alpha = 3}^N \left( \frac{x^{\frac{s+1-k(\alpha-2)}{2}}}{(\log x)^{s+1-\alpha k}} + x^{(s+1)/3} \right) \\
   & \ll \frac{N^\ell x^{\frac{s+1-k}{2}}}{(\log x)^{s+2-3k}} + \frac{N^\ell x^{(s+1)/3}}{\log x} \ll \frac{x^{(s+1)/2}}{(\log x)^{s+2-2k}}
\end{align*}
where we used the facts that the function $\alpha \mapsto \frac{x^{\frac{s+1-k(\alpha-2)}{2}}}{(\log x)^{s+1-\alpha k}}$ is non-increasing and $N \leqslant 7$.
\end{proof}

\subsection{Series related to the Riemann zeta function}

\begin{lem}
\label{le:k_qq}
Let $m \in \Z_{\geqslant 0}$ and $\ell \in \Z_{\geqslant 0}$ such that $m - \ell \geqslant 2$. For all $k \in \Z_{\geqslant 1}$ such that $m - k - \ell \geqslant 1$, we have
\begin{equation}
   \sum_{n_1=1}^\infty \dotsb \sum_{n_k = 1}^\infty \min \left( \frac{1}{n_1^m} , \dotsc , \frac{1}{n_k^m} \right) \max \left( n_1^\ell, \dotsc,n_k^\ell \right) = \sum_{j=0}^{k-1} (-1)^{k-1-j} {k \choose j} \zeta(m-\ell-j). \label{eq:serie_zeta}
\end{equation}
\end{lem}

\begin{proof}
We first show how the condition $m-k - \ell \geqslant 1$ ensures the convergence of the series. Using the harmonic mean we get
$$\min \left( \frac{1}{n_1^m} , \dotsc , \frac{1}{n_k^m} \right) \max \left( n_1^\ell, \dotsc,n_k^\ell \right) \leqslant \frac{k}{n_1^m + \dotsb + n_k^m} \times \left( \max \left( n_1, \dotsc,n_k \right) \right)^{\ell}$$
and H\"{o}lder's inequality yields
$$\min \left( \frac{1}{n_1^m} , \dotsc , \frac{1}{n_k^m} \right) \max \left( n_1^\ell, \dotsc,n_k^\ell \right) \leqslant k^m \frac{\left( n_1 + \dotsb + n_k \right)^\ell}{\left( n_1 + \dotsb + n_k \right)^m}.$$
Since $m - k - \ell \geqslant 1$, we see that the left-hand side is
$$\leqslant k^m \frac{\left( n_1 + \dotsb + n_k \right)^\ell}{\left( n_1 + \dotsb + n_k \right)^{\ell+k+1}} = \frac{k^m}{\left( n_1 + \dotsb + n_k \right)^{k+1}} \leqslant \frac{k^{m-k-1}}{\left( n_1 \dotsb n_k \right)^{1+1/k}},$$
where we used the arithmetic-geometric mean inequality in the last step, implying the convergence of the series. Now set $S_{k,\ell}$ its sum. Since
$$\min \left( \frac{1}{n_1^m} , \dotsc , \frac{1}{n_k^m} \right) \leqslant \frac{1}{n_{k+1}^m} \iff n_{k+1} \leqslant \mx \left( n_1, \dotsc,n_k \right)$$
we derive, assuming $m - k - \ell \geqslant 2$
\begin{align*}
   S_{k+1,\ell} &= \sum_{n_1=1}^\infty \dotsb \sum_{n_k = 1}^\infty \sum_{n_{k+1}=1}^\infty \min \left\lbrace \min \left( \frac{1}{n_1^m} , \dotsc , \frac{1}{n_k^m} \right) \, , \, \frac{1}{n_{k+1}^m} \right\rbrace \mx \left\lbrace \mx \left( n_1^\ell, \dotsc,n_k^\ell \right) \, , \, n_{k+1}^\ell \right\rbrace \\
   &= \sum_{n_1=1}^\infty \dotsb \sum_{n_k = 1}^\infty \min \left( \frac{1}{n_1^m} , \dotsc , \frac{1}{n_k^m} \right) \mx \left( n_1^\ell, \dotsc,n_k^\ell \right) \sum_{n_{k+1} = 1}^{\mx \left( n_1, \dotsc,n_k \right)} 1 \\
   & \qquad \qquad +  \sum_{n_1=1}^\infty \dotsb \sum_{n_k = 1}^\infty \ \sum_{n_{k+1} = \mx \left( n_1, \dotsc,n_k \right)+1}^\infty \frac{n_{k+1}^\ell}{n_{k+1}^m} \\
   &= \sum_{n_1=1}^\infty \dotsb \sum_{n_k = 1}^\infty \min \left( \frac{1}{n_1^m} , \dotsc , \frac{1}{n_k^m} \right) \mx \left( n_1^{\ell+1}, \dotsc,n_k^{\ell+1} \right) \\
   & \qquad \qquad + \sum_{n_{k+1} = 1}^\infty \frac{1}{n_{k+1}^{m-\ell}} \ \sum_{n_1=1}^{n_{k+1}-1} \dotsb \sum_{n_k=1}^{n_{k+1}-1} 1 \\
   &= S_{k,\ell+1} + \sum_{n_{k+1} = 1}^\infty \frac{\left( n_{k+1} - 1 \right)^k}{n_{k+1}^{m-\ell}} = S_{k,\ell+1} + \sum_{n_{k+1} = 1}^\infty \frac{1}{n_{k+1}^{m-\ell}} \sum_{j=0}^k (-1)^{k-j} {k \choose j} n_{k+1}^j \\
   &= S_{k,\ell+1} + \sum_{j=0}^k (-1)^{k-j} {k \choose j} \zeta(m-\ell - j).
\end{align*}
We now prove the result by induction on $k$. When $k=1$, both sides of \eqref{eq:serie_zeta} are equals to $\zeta(m-\ell)$. Next assuming that \eqref{eq:serie_zeta} holds for some $k$ satisfying $m - k - \ell \geqslant 2$, we get
\begin{align*}
   S_{k+1,\ell} &= \sum_{j=0}^{k-1} (-1)^{k-1-j} {k \choose j} \zeta(m-\ell-1-j) + \sum_{j=0}^k (-1)^{k-j} {k \choose j} \zeta(m-\ell - j) \\
   &= \sum_{j=1}^{k} (-1)^{k-j} {k \choose j-1} \zeta(m-\ell-j) + \sum_{j=0}^k (-1)^{k-j} {k \choose j} \zeta(m-\ell - j) \\
   &= \sum_{j=0}^{k} (-1)^{k-j} {k \choose j-1} \zeta(m-\ell-j) + \sum_{j=0}^k (-1)^{k-j} {k \choose j} \zeta(m-\ell - j) \\
   &= \sum_{j=0}^k (-1)^{k-j} \left( {k \choose j} + {k \choose j-1} \right)\zeta(m-\ell - j)  \\
   &= \sum_{j=0}^k (-1)^{k-j}  {k +1 \choose j}\zeta(m-\ell - j)
\end{align*}
completing the proof.
\end{proof}

\begin{coro}
\label{cor:k_qq}
Let $m \in \Z_{\geqslant 2}$. For all $k \in \Z_{\geqslant 1}$ such that $m-k \geqslant 1$, we have
$$\sum_{n_1=1}^\infty \dotsb \sum_{n_k = 1}^\infty \min \left( \frac{1}{n_1^m} , \dotsc , \frac{1}{n_k^m} \right) = \sum_{j=0}^{k-1} (-1)^{k-1-j} {k \choose j} \zeta(m-j).$$
\end{coro}

\begin{proof}
Follows from Lemma~\ref{le:k_qq} with $\ell = 0$.
\end{proof}

\begin{lem}
\label{le:remainder}
Let $m \in \Z_{\geqslant 2}$ and $k \in \Z_{\geqslant 1}$ such that $m-k \geqslant 1$. Define
$$\Xi_{k,m} := \sum_{n_1=1}^\infty \dotsb \sum_{n_k = 1}^\infty \min \left( \frac{1}{n_1^m} , \dotsc , \frac{1}{n_k^m} \right).$$
Then, for all $x > e$
$$\left| \Xi_{k,m} - \sum_{n_1 \leqslant \log x} \dotsb \sum_{n_k \leqslant \log x} \min \left( \frac{1}{n_1^m} , \dotsc , \frac{1}{n_k^m} \right) \right| \ll (\log x)^{1-\frac{m}{k}}.$$
\end{lem}

\begin{proof}
We have
\begin{multline*}
   \left| \Xi_{k,m} - \sum_{n_1 \leqslant \log x} \dotsb \sum_{n_k \leqslant \log x} \min \left( \frac{1}{n_1^m} , \dotsc , \frac{1}{n_k^m} \right) \right| \\
   \leqslant \sum_{j=1}^k {k \choose j} \sum_{n_1=1}^\infty \dotsb \sum_{n_{k-j} = 1}^\infty \ \sum_{n_{k-j+1}> \log x} \dotsb \sum_{n_k> \log x} \min \left( \frac{1}{n_1^m} , \dotsc , \frac{1}{n_k^m} \right)
\end{multline*}
and the inequality $\min \left( \frac{1}{n_1^m} , \dotsc , \frac{1}{n_k^m} \right) \leqslant (n_1 \dotsb n_k)^{-m/k}$ yields
\begin{align*}
   & \left| \Xi_{k,m} - \sum_{n_1 \leqslant \log x} \dotsb \sum_{n_k \leqslant \log x} \min \left( \frac{1}{n_1^m} , \dotsc , \frac{1}{n_k^m} \right) \right| \\
   & \qquad \leqslant \sum_{j=1}^k {k \choose j} \sum_{n_1=1}^\infty \frac{1}{n_1^{m/k}} \dotsb \sum_{n_{k-j} = 1}^\infty \frac{1}{n_{k-j}^{m/k}} \sum_{n_{k-j+1}> \log x} \frac{1}{n_{k-j+1}^{m/k}} \dotsb \sum_{n_k> \log x}  \frac{1}{n_k^{m/k}} \\
   & \qquad \leqslant \sum_{j=1}^k {k \choose j} \zeta \left (\frac{m}{k} \right )^{k-j} \left( \frac{1}{(\frac{m}{k}-1) (\log x)^{m/k-1}} + \frac{1}{(\log x)^{m/k}} \right)^j \\
   & \qquad = \left( \frac{1}{(\frac{m}{k}-1) (\log x)^{m/k-1}} + \frac{1}{(\log x)^{m/k}} + \zeta \left (\frac{m}{k} \right ) \right)^k - \zeta \left (\frac{m}{k} \right )^k \\
   & \qquad \ll \frac{1}{(\log x)^{m/k-1}}
\end{align*}
as required.
\end{proof}

\subsection{A sum of certain fractional parts}

The next result is an extension of \cite[Th\'{e}or\`{e}me~1]{mer87}. Note that we could have followed straightforwardly the proof of Mercier's result, but we think there is a flaw in it. Indeed, in \cite[page 312, line $4$]{mer87}, the author claims that
$$\int_{2^-}^{x^+} y^a \left\lbrace x/y \right\rbrace^k \, \textrm{d} \left( \pi(y) - \Li(y) \right) \ll \int_{2^-}^{x^+} y^a  \, \textrm{d} \left( \pi(y) - \Li(y) \right)$$
but this seems to be incorrect for the integrator $\pi(y) - \Li(y)$ is not monotonically increasing. Below we propose a modification of Mercier's proof to provide a clarification of this defected argument.

\begin{prop}
\label{pro:frac_parts}
Let $j,\ell \in \Z_{\geqslant 0}$ and $\alpha \in \Z_{\geqslant 1}$. Then, for all $N \in \Z_{\geqslant 1}$
$$\sum_{p \leqslant x^{1/\alpha}}  p^\ell \left\lbrace \frac{x}{p^\alpha} \right\rbrace^j = \frac{x^{\frac{\ell+1}{\alpha}}}{\log x} \sum_{h=0}^{N-1} \frac{a_{j,h,\ell,\alpha}}{(\log x)^h} + O \left( \frac{x^{\frac{\ell+1}{\alpha}}}{(\log x)^{N+1}} \right)$$
where 
\begin{equation}
   a_{j,h,\ell,\alpha} := \int_1^\infty \frac{\{u\}^j (\log u)^h}{u^{1+(\ell+1)/\alpha}} \, \mathrm{d}u. \label{eq:nombres a_j,h,l,alpha}
\end{equation}
The error term depends at most on $j,\ell,\alpha,N$.
\end{prop}

\begin{proof}
First some specific notation we will use here. Set $y = x^{1/\alpha}$ and, for $N \in \Z_{\geqslant 0}$, define
\begin{equation}
   T := \frac{x^{1/\alpha}}{(\log x)^{N/(\ell+1)}} = \left( \frac{x}{L} \right)^{1/\alpha} \label{eq:T}
\end{equation}
where $L:=(\log x)^{\frac{N \alpha}{\ell+1}}$. Recall that we assume that $x$ is sufficiently large to ensure the inequality $x \geqslant L^2$. Note that this also implies $T \geqslant x^{\frac{1}{2 \alpha}}$. Now write
\begin{align*}
   S_{\alpha,j,\ell} (x) &:= \left( \sum_{p \leqslant T} + \sum_{T < p \leqslant y} \right) p^\ell \left\lbrace \frac{x}{p^\alpha} \right\rbrace^j \\
   &= \sum_{T < p \leqslant y} p^\ell \left\lbrace \frac{x}{p^\alpha} \right\rbrace^j + O \left( \frac{T^{\ell+1}}{\log T} \right) \\
   &:= S_T(x) + O_{\alpha} \left( \frac{T^{\ell+1}}{\log x} \right).
\end{align*}
Next, setting $R_\ell(t) := \pi_\ell(t) - \Lil(t)$, we derive
\begin{align*}
   S_T(x) &= \int_{T^-}^{y} \left\lbrace \frac{x}{t^\alpha} \right\rbrace^j \, \textrm{d} \pi_\ell(t) \notag \\
   &= \int_T^y \frac{t^\ell}{\log t} \left\lbrace \frac{x}{t^\alpha} \right\rbrace^j \, \textrm{d} t + \int_{T^-}^{y} \left\lbrace \frac{x}{t^\alpha} \right\rbrace^j \, \textrm{d} R_\ell(t) \\
   &= x^{\frac{\ell+1}{\alpha}} \int_1^L \frac{\{u\}^j}{u^{1+\frac{\ell+1}{\alpha}} \log(x/u)} \, \textrm{d}u + I_T(x). \notag
\end{align*}

Let us verify that the integral $I_T(x)$ does exist. A point $t \in \left[ T,y \right]$ is a point of discontinuity of the function $t \mapsto \left\lbrace x/t^\alpha \right\rbrace^j$ if there exists $m \in \Z$ such that $x/t^\alpha = m$, which is equivalent to 
$$t = \left( \frac{x}{m} \right)^{1/\alpha} \ \textrm{and} \ m \in \left[ 1,L \right] \cap \Z.$$
The points of discontinuity of the function $R_\ell$ are the prime numbers belonging to the interval $\left[ T,y \right]$. Hence, these two functions have a point of discontinuity in common whenever
$$\exists \, m \in \left[ 1,L \right] \cap \Z, \ \;  \exists \, q \in \left[ T,y \right] \cap \p, \ \;  x = mq^{\alpha}.$$
In particular, $I_T(x)$ exists if $x \not \in \Z$. Now assume $x \in \Z$, and let $q \in \left[ T,y \right] \cap \p$ be a point of discontinuity in common to these two functions, so that $x = mq^{\alpha}$ for some integer $m \in \left[ 1,L \right]$. Observe that the function $t \mapsto \left\lbrace x/t^\alpha \right\rbrace^j = \left\lbrace m(q/t)^\alpha \right\rbrace^j$ is left-continuous at $q$, whereas the function $R_\ell$ is right-continuous at $q$. By \cite[Theorem~5.6.4]{mon19}, we deduce that $I_T(x)$ also exists in this case.

Denoting $V_a^b(F)$ the total variation on $\left[ a,b \right]$ of a function $F$ of bounded variation, and applying \cite[inequality (12) p. 77]{kar49} to the functions $f_x^j(t) := \left\lbrace x/t^\alpha \right\rbrace^j$ and $R_\ell$, we derive
\begin{align*}
   \left| I_T(x) \right| & \leqslant \left( 1 + V_T^y \left( f_x^j \right) \right) \times \sup_{T \leqslant t \leqslant y} \Bigl( \bigl| R_\ell(t) - R_\ell(T) \bigr| \Bigr) \\
   & \leqslant \left( 1 + j V_T^y \left( f_x \right) \right) \times \sup_{T \leqslant t \leqslant y} \Bigl( \bigl| R_\ell(t) - R_\ell(T) \bigr| \Bigr) \\
   & \ll_{j,\ell} \; \left( 1 + \frac{x}{T^\alpha} \right)  \times y^{\ell+1} \delta_c(y) \ll_{j,\ell} \; x^{1 + \frac{\ell+1}{\alpha}} T^{- \alpha} \delta_c \left( x^{1/\alpha} \right)
\end{align*}
where we used the inequality $V_a^b \left( F^j \right) \leqslant j \|F\|_{\infty}^{j-1} V_a^b(F)$ in the penultimate line and Lemma~\ref{le:sum_p_beta} in the last one. The choice \eqref{eq:T} of $T$ then yields
\begin{equation}
   \sum_{p \leqslant x^{1/\alpha}} p^\ell \left\lbrace \frac{x}{p^\alpha} \right\rbrace^j = x^{\frac{\ell+1}{\alpha}} \int_1^L \frac{\{u\}^j}{u^{1+\frac{\ell+1}{\alpha}} \log(x/u)} \, \textrm{d}u + O_{\alpha,j,\ell} \left( \frac{x^{\frac{\ell+1}{\alpha}}}{(\log x)^{N+1}}\right). \label{eq:I_0}
\end{equation}
Now the main term is equal to
\begin{align*}
   x^{\frac{\ell+1}{\alpha}} \int_1^L \frac{\{u\}^j}{u^{1+\frac{\ell+1}{\alpha}} \log(x/u)} \, \textrm{d}u &= \frac{x^{\frac{\ell+1}{\alpha}}}{\log x} \sum_{h=0}^{N-1} \frac{1}{(\log x)^h} \ \int_1^L \frac{\{u\}^j (\log u)^h}{u^{1+\frac{\ell+1}{\alpha}}} \, \textrm{d}u \\
   & \qquad + \frac{x^{\frac{\ell+1}{\alpha}}}{(\log x)^{N+1}} \int_1^L \frac{\{u\}^j (\log u)^N}{u^{1+\frac{\ell+1}{\alpha}}} \left( 1 - \frac{\log u}{\log x}\right)^{-1} \, \textrm{d}u
\end{align*}
with
\begin{align}
   & \frac{x^{\frac{\ell+1}{\alpha}}}{\log x} \sum_{h=0}^{N-1} \frac{1}{(\log x)^h} \ \int_L^\infty \frac{\{u\}^j (\log u)^h}{u^{1+\frac{\ell+1}{\alpha}}} \, \textrm{d}u \notag \\
   & \leqslant \frac{\alpha}{\ell+1} \frac{x^{\frac{\ell+1}{\alpha}}}{\log x} \sum_{h=0}^{N-1} \left( h! \left( 1 + \frac{\alpha}{\ell+1} \right)^h \frac{1}{(\log x)^h} \times \frac{(\log x)^h}{(\log x)^{\frac{N \alpha}{\ell+1} \times \frac{\ell+1}{\alpha}}} \right) \notag \\
   & = \frac{\alpha}{\ell+1} \frac{x^{\frac{\ell+1}{\alpha}}}{(\log x)^{N+1}} \sum_{h=0}^{N-1} \left( h! \left( 1 + \frac{\alpha}{\ell+1} \right)^h \right) < \frac{N! \, \alpha}{\ell+1} \frac{x^{\frac{\ell+1}{\alpha}}}{(\log x)^{N+1}} \sum_{h=0}^{N-1} \left( 1 + \frac{\alpha}{\ell+1} \right)^h \notag \\
   &  = \frac{N! \, \alpha}{\ell+1} \frac{x^{\frac{\ell+1}{\alpha}}}{(\log x)^{N+1}} \times \frac{\left( 1 + \frac{\alpha}{\ell+1} \right)^N-1}{\frac{\alpha}{\ell+1}} < N! \left( 1 + \frac{\alpha}{\ell+1} \right)^N \, \frac{x^{\frac{\ell+1}{\alpha}}}{(\log x)^{N+1}} \label{eq:I_1}
\end{align}
and
\begin{align}
   & \frac{x^{\frac{\ell+1}{\alpha}}}{(\log x)^{N+1}} \int_1^L \frac{\{u\}^j (\log u)^N}{u^{1+\frac{\ell+1}{\alpha}}} \left( 1 - \frac{\log u}{\log x}\right)^{-1} \, \textrm{d}u \notag \\
   & \leqslant \frac{2x^{\frac{\ell+1}{\alpha}}}{(\log x)^{N+1}} \int_1^\infty \frac{(\log u)^N}{u^{1+\frac{\ell+1}{\alpha}}}\, \textrm{d}u = 2N! \left( \frac{\alpha}{\ell+1} \right)^{N+1} \, \frac{x^{\frac{\ell+1}{\alpha}}}{(\log x)^{N+1}}. \label{eq:I_2}
\end{align}

Inserting \eqref{eq:I_1} and \eqref{eq:I_2} in \eqref{eq:I_0} allows us to complete the proof.
\end{proof}

\subsection{Sums of primes}

\begin{prop}
\label{pro:sum_primes_1}
Let $r \in \Z_{\geqslant 0}$ and $k \in \Z_{\geqslant 1}$.
\begin{enumerate}
   \item[\scriptsize $\triangleright$] If $k \geqslant r+2$, then
$$\sum_{p \leqslant x} p^r \left \lfloor \frac{x}{p} \right \rfloor^k = x^k \sum_p \frac{1}{p^{k-r}} + O \left( x^{k-1} (\log \log x)^{\kappa_{r,k-2}} \right).$$
   \item[\scriptsize $\triangleright$] If $k=r+1$, then, for all $N \in \Z_{\geqslant 1}$
$$\sum_{p \leqslant x} p^{k-1} \left \lfloor \frac{x}{p} \right \rfloor^k = x^k \log \log x + M x^k + \frac{x^k}{\log x} \sum_{h=0}^{N-1} \frac{A_{k,h}}{(\log x)^h} + O \left( \frac{x^k}{(\log x)^{N+1}} \right)$$
where $M$ is defined in \eqref{eq:Mertens_cst} and 
\begin{equation}
   A_{k,h} := \sum_{j=1}^k (-1)^j {k \choose j} \int_1^\infty \frac{\{u\}^j (\log u)^h}{u^{j+1}}  \, \mathrm{d}u. \label{eq:A_{k,h}}
\end{equation}
   \item[\scriptsize $\triangleright$] If $k \leqslant r$, then
   $$\sum_{p \leqslant x} p^r \left \lfloor \frac{x}{p} \right \rfloor^k = \frac{x^{r+1}}{(r+1)\log x} \sum_{j=0}^{k-1} (-1)^{k-1-j} {k \choose j} \zeta(r+1-j) + O \left(    x^{r+1} \left( (\log x)^{-\frac{r+1}{k}} + (\log x)^{-2} \right) \right).$$
\end{enumerate}
\end{prop}

\begin{proof}
If $k \geqslant r+2$
\begin{align*}
   \sum_{p \leqslant x} p^r \left \lfloor \frac{x}{p} \right \rfloor^k &= \sum_{p \leqslant x} p^r \left\lbrace \frac{x^k}{p^k} + O \left( \frac{x^{k-1}}{p^{k-1}}\right) \right\rbrace \\
   &= x^k \sum_{p \leqslant x} \frac{1}{p^{k-r}} + O \left( x^{k-1} \sum_{p \leqslant x} \frac{1}{p^{k-1-r}} \right) \\
   &= x^k \sum_p \frac{1}{p^{k-r}} - x^k \sum_{p>x} \frac{1}{p^{k-r}} + O \left( x^{k-1} (\log \log x)^{\kappa_{r,k-2}} \right) \\
   &= x^k \sum_p \frac{1}{p^{k-r}} + O \left( x^{r+1} (\log x)^{-1} \right) + O \left( x^{k-1} (\log \log x)^{\kappa_{r,k-2}} \right)
\end{align*}
and we conclude by noticing that $r+1 \leqslant k-1$ by assumption. If $k=r+1$, writing $\left \lfloor x/p \right \rfloor = x/p - \{x/p \}$ and using Newton's formula, we derive
\begin{align*}
    \sum_{p \leqslant x} p^{k-1} \left \lfloor \frac{x}{p} \right \rfloor^k &= \sum_{j=0}^k (-1)^j {k \choose j} x^{k-j} \sum_{p \leqslant x} p^{j-1} \left\lbrace \frac{x}{p} \right\rbrace^j \\
    &= x^{k} \sum_{p \leqslant x} \frac{1}{p} + \sum_{j=1}^k (-1)^j {k \choose j} x^{k-j} \sum_{p \leqslant x} p^{j-1} \left\lbrace \frac{x}{p} \right\rbrace^j
\end{align*}
and using Proposition~\ref{pro:frac_parts} with $\alpha = 1$ and $\ell = j-1$, $j \geqslant 1$, yields
$$\sum_{p \leqslant x} p^{k-1} \left \lfloor \frac{x}{p} \right \rfloor^k = x^k \left( \log \log x + M + O \left( \delta_c(x)\right) \right) + \frac{x^k}{\log x} \sum_{h=0}^{N-1} \frac{1}{(\log x)^h} \sum_{j=1}^k (-1)^j {k \choose j} a_{j,h,j-1,1} + O \left( \frac{x^k}{(\log x)^{N+1}} \right)$$
for some $c > 0$, implying the asserted result. The case $k \leqslant r$ is more difficult. Using Lemmas~\ref{le:transformation_case_1},~\ref{le:useful_est} and~\ref{le:sum_p_beta}, we derive
\begin{align*}
   \sum_{p \leqslant x} p^r \left \lfloor \frac{x}{p} \right \rfloor^k &= \left( \sum_{p \leqslant \frac{x}{\log x}} + \sum_{\frac{x}{\log x} < p \leqslant x} \right)  p^r \left \lfloor \frac{x}{p} \right \rfloor^k \\
   &= \sum_{n_1 < \log x} \dotsb \sum_{n_k < \log x} \ \sum_{p \leqslant M_k x} p^r + O \left( \frac{x^{r+1}}{(\log x)^{r-k+2}}\right) \\
   &= \frac{x^{r+1}}{r+1} \sum_{n_1 < \log x} \dotsb \sum_{n_k < \log x} \frac{M_k^{r+1}}{\log (M_k x)} + O \left( \frac{x^{r+1}}{(\log x)^2} + \frac{x^{r+1}}{(\log x)^{r-k+2}} \right)  \\
   &= \frac{x^{r+1}}{(r+1)\log x} \sum_{n_1 < \log x} \dotsb \sum_{n_k < \log x} M_k^{r+1} \\
   & \qquad + O \left( \frac{x^{r+1}}{(\log x)^2} + \frac{x^{r+1}}{(\log x)^{r-k+2}} + \frac{x^{r+1}}{(\log x)^2} \sum_{n_1 < \log x} \dotsb \sum_{n_k < \log x} M_k^{r+1} \left| \log M_k \right| \right) 
\end{align*}
the $2$nd term being absorbed by the $1$st one since $k \leqslant r$, and where we used
$$\frac{1}{\log (M_k x)} = \frac{1}{\log x} \left( 1 + O \left( \frac{\left| \log M_k \right|}{\log x}\right) \right)$$
in the last error term. Since $\left| \log M_k \right| = \log \mx \left (n_1, \dotsc, n_k \right)$, the series $\sum_{n_1 \geqslant 1} \dotsb \sum_{n_k \geqslant 1} M_k^{r+1} \left| \log M_k \right|$ converges by Lemma~\ref{le:k_qq}, so that
\begin{align}
   \sum_{p \leqslant x} p^r \left \lfloor \frac{x}{p} \right \rfloor^k &= \frac{x^{r+1}}{(r+1)\log x} \sum_{n_1 < \log x} \dotsb \sum_{n_k < \log x} M_k^{r+1} + O \left( \frac{x^{r+1}}{(\log x)^2} \sum_{n_1 < \log x} \dotsb \sum_{n_k < \log x} M_k^{r+1} \right) \notag \\
   &=  \frac{x^{r+1}}{(r+1)\log x} \sum_{n_1 < \log x} \dotsb \sum_{n_k < \log x} M_k^{r+1} + O \left( \frac{x^{r+1}}{(\log x)^2} \right). \label{eq:est_almost_closed}
\end{align}
Now the result follows by inserting in \eqref{eq:est_almost_closed} the results of Corollary~\ref{cor:k_qq} and Lemma~\ref{le:remainder} with $m=r+1$ whose error term is multiplied out by $\frac{x^{r+1}}{\log x}$.
\end{proof}

\begin{prop}
\label{pro:sum_primes_2}
Let $s,\ell \in \Z_{\geqslant 0}$, $k \in \Z_{\geqslant 1}$ and $g_\ell$ be a function satisfying \eqref{eq:g_l(alpha)}.
\begin{enumerate}
   \item[\scriptsize $\triangleright$] If $s \leqslant 2k-3$, then
   $$\sum_{\substack{p^\alpha \leqslant x \\ \alpha \geqslant 2}} g_\ell(\alpha) p^s \left \lfloor \frac{x}{p^\alpha} \right \rfloor^k = x^k \sum_p \ \sum_{\alpha=2}^\infty \frac{g_\ell(\alpha)}{p^{\alpha k - s}} + O \left( x^{k-1} (\log x)^{\ell} (\log \log x)^{\kappa_{s,2k-3} \times \kappa_{\ell,0}} \right).$$
   \item[\scriptsize $\triangleright$] If $s = 2k-2$, then
   $$\sum_{\substack{p^\alpha \leqslant x \\ \alpha \geqslant 2}} g_\ell(\alpha) p^{2(k-1)} \left \lfloor \frac{x}{p^\alpha} \right \rfloor^k = x^k \sum_p \ \sum_{\alpha=2}^\infty \frac{g_\ell(\alpha)}{p^{\alpha (k-2) + 2}} + O \left( \frac{x^{k-1/2}}{\log x} \right).$$
   \item[\scriptsize $\triangleright$] If $s  = 2k-1$, then, for all $N \in \Z_{\geqslant 2}$
   $$\sum_{\substack{p^\alpha \leqslant x \\ \alpha \geqslant 2}} g_\ell(\alpha) p^{2k-1} \left \lfloor \frac{x}{p^\alpha} \right \rfloor^k = g_\ell(2) x^k \log\log \sqrt{x} + D_k x^k + \frac{g_\ell(2) x^k}{\log x} \sum_{h=0}^{N-1} \frac{A_{k,h}}{(\log x)^h} + O \left( \frac{x^k}{(\log x)^{N+1}} \right)$$
where the numbers $A_{k,h}$ are given in \eqref{eq:A_{k,h}} and
\begin{equation}
   D_k := g_\ell(2) M + \sum_p \ \sum_{\alpha = 3}^\infty \frac{g_\ell(\alpha)}{p^{1+k(\alpha-2)}}. \label{eq:D_k}
\end{equation}
   \item[\scriptsize $\triangleright$] If $s \geqslant 2k$, then
   \begin{multline*}
   \sum_{\substack{p^\alpha \leqslant x \\ \alpha \geqslant 2}} g_\ell(\alpha) p^s \left \lfloor \frac{x}{p^\alpha} \right \rfloor^k = \frac{2g_\ell(2)x^{(s+1)/2}}{(s+1)\log x} \sum_{j=0}^{k-1} (-1)^{k-1-j} {k \choose j} \zeta \left( \tfrac{s+1}{2}-j \right) \\
   + O \left( x^{(s+1)/2} \left( (\log x)^{1-\frac{s+1}{k}} + (\log x)^{-2} \right) \right).
   \end{multline*}
\end{enumerate}
\end{prop}

\begin{proof}
Assume first $s \leqslant 2k-2$.
\begin{align*}
   \sum_{\substack{p^\alpha \leqslant x \\ \alpha \geqslant 2}} g_\ell(\alpha) p^s \left \lfloor \frac{x}{p^\alpha} \right \rfloor^k &= \sum_{p \leqslant \sqrt{x}} p^s \sum_{2 \leqslant \alpha \leqslant \frac{\log x}{\log p}} g_\ell(\alpha) \left( \frac{x^k}{p^{\alpha k}} + O \left( \frac{x^{k-1}}{p^{\alpha(k-1)}}\right) \right) \\
   &= x^k \sum_{p \leqslant \sqrt{x}} \ \sum_{2 \leqslant \alpha \leqslant \frac{\log x}{\log p}} \frac{g_\ell(\alpha)}{p^{\alpha k-s}} + O \left( x^{k-1} \sum_{p \leqslant \sqrt{x}} \ \sum_{2 \leqslant \alpha \leqslant \frac{\log x}{\log p}} \frac{\left| g_\ell(\alpha) \right|}{p^{\alpha (k-1) - s}}\right).
\end{align*}
The error term does not exceed
\begin{align*}
   & \ll x^{k-1} \sum_{p \leqslant \sqrt{x}} p^s \sum_{2 \leqslant \alpha \leqslant \frac{\log x}{\log p}} \frac{\alpha^\ell}{p^{\alpha(k-1)}} \ll x^{k-1} (\log x)^\ell \sum_{p \leqslant \sqrt{x}} \frac{p^{s+1-k}}{(\log p)^\ell (p^{k-1} - 1)} \\
   & \ll x^{k-1} (\log x)^\ell \sum_{p \leqslant \sqrt{x}} \frac{1}{(\log p)^\ell p^{2k-2-s}} := x^{k-1} (\log x)^\ell R_{k,\ell,s}(x)
\end{align*}
where
$$R_{k,\ell,s}(x) \ll \begin{cases} 1, & \textrm{if\ } s \leqslant 2k-3 \ \textrm{and} \ \ell \geqslant 1 \, ;  \\ (\log \log x)^{\kappa_{s,2k-3}} , & \textrm{if\ } s \leqslant 2k-3 \ \textrm{and} \ \ell = 0 \, ; \\ \sum_{p \leqslant \sqrt{x}} \frac{1}{(\log p)^\ell}  \ll \frac{\pi \left( \sqrt{x} \right)}{(\log x)^\ell} & \textrm{if\ } s = 2k-2 . \end{cases}$$
The main term is written as 
$$x^k \left( \sum_p \ \sum_{\alpha = 2}^\infty \frac{g_\ell(\alpha)}{p^{\alpha k-s}} - \sum_{p> \sqrt{x}} \ \sum_{\alpha = 2}^\infty \frac{g_\ell(\alpha)}{p^{\alpha k-s}}- \sum_{p \leqslant \sqrt{x}} \ \sum_{\alpha > \frac{\log x}{\log p}} \frac{g_\ell(\alpha)}{p^{\alpha k-s}}\right)$$
and using Lemma~\ref{le:toth} with $z=2$, assuming $x \geqslant e^{\frac{2 \ell+1}{k}}$, we derive
\begin{align*}
   \left| \sum_{p> \sqrt{x}} \ \sum_{\alpha = 2}^\infty \frac{g_\ell(\alpha)}{p^{\alpha k-s}} \right| & \leqslant C_0 \sum_{p> \sqrt{x}} p^s \sum_{\alpha=2}^\infty \frac{\alpha^\ell}{p^{\alpha k}} = C_0 \sum_{p> \sqrt{x}} p^s \left\lbrace \frac{2^\ell}{p^{2k}} + \sum_{\alpha=3}^\infty \frac{\alpha^\ell}{p^{\alpha k}} \right\rbrace \\
   & \leqslant 4^\ell C_0 \sum_{p> \sqrt{x}} \frac{1}{p^{2k-s}} \ll_{k,\ell,C_0} \frac{x^{\frac{s+1}{2}}}{x^k \log x},
\end{align*}
and, using Lemma~\ref{le:toth} again with $z = \frac{\log x}{\log p}$, assuming $x \geqslant 2^{\frac{\ell}{k \log 2 - 1/2}}$, we get
\begin{align*}
   \left| \sum_{p \leqslant \sqrt{x}} \ \sum_{\alpha > \frac{\log x}{\log p}} \frac{g_\ell(\alpha)}{p^{\alpha k-s}} \right| & \leqslant C_0 \sum_{p \leqslant \sqrt{x}} p^s \sum_{\alpha > \frac{\log x}{\log p}} \frac{\alpha^\ell}{p^{\alpha k}} \\
   & \leqslant \frac{3 C_0 (\log x)^\ell}{x^k} \sum_{p \leqslant \sqrt{x}} \frac{p^s}{(\log p)^\ell} \ll_{k,\ell,C_0} \frac{x^{\frac{s+1}{2}}}{x^k \log x}.
\end{align*}
If $s=2k-1$, as in Proposition~\ref{pro:sum_primes_1}, writing $\left \lfloor x/p^\alpha \right \rfloor = x/p^\alpha - \{x/p^\alpha \}$ and using Newton's formula, we derive
$$\sum_{\substack{p^\alpha \leqslant x \\ \alpha \geqslant 2}} g_\ell(\alpha)p^{2k-1} \left \lfloor \frac{x}{p^\alpha} \right \rfloor^k = \sum_{j=0}^k (-1)^j {k \choose j} x^{k-j} \sum_{\alpha = 2}^{L_1} g_\ell(\alpha) \sum_{p \leqslant x^{1/\alpha}} p^{2k-1 - \alpha (k-j)} \left\lbrace \frac{x}{p^\alpha} \right\rbrace^j$$
where $L_1 := L_1(x) = \left \lfloor \frac{\log x}{\log 2} \right \rfloor$. Now notice that, when $j <k$ and $\alpha > \frac{2k}{k-j}$, then $2k-1 - \alpha (k-j) < -1$, so that using $\left \lfloor \frac{2k}{k-j} \right \rfloor \geqslant \frac{2k+1}{k-j} - 1$ and assuming $x \geqslant 2 \times 4^k$ to ensure that $L_1 > \frac{2k}{k-j}$, we get
\begin{align*}
   & \left| \sum_{j=1}^{k-1} (-1)^j {k \choose j} x^{k-j} \sum_{\frac{2k}{k-j} < \alpha \leqslant L_1} g_\ell(\alpha) \sum_{p \leqslant x^{1/\alpha}} p^{2k-1 - \alpha (k-j)} \left\lbrace \frac{x}{p^\alpha} \right\rbrace^j \right| \\
   & \qquad \leqslant C_0 \sum_{j=1}^{k-1} {k \choose j} x^{k-j} \sum_{\frac{2k+1}{k-j} \leqslant \alpha \leqslant L_1} \alpha^\ell \sum_{p \leqslant x^{1/\alpha}} \frac{p^{2k-1}}{p^{\alpha(k-j)}} \\
   & \qquad \leqslant C_0 L_1^{\ell+1} \left( \sum_p \frac{1}{p^2} \right) \sum_{j=1}^{k-1} {k \choose j} x^{k-j}  \\
   & \qquad \leqslant C_0 L_1^{\ell+1} \left( \sum_p \frac{1}{p^2} \right) (x+1)^{k-1} \ll_{k,\ell,C_0} x^{k-1} (\log x)^{\ell+1}.
\end{align*} 
Also note that, if $\frac{2k}{k-j} \not \in \Z$, then $\left \lfloor \frac{2k}{k-j} \right \rfloor \leqslant \frac{2k-1}{k-j}$, so that the interval $\left( \frac{2k-1}{k-j} \, , \, \frac{2k}{k-j} \right]$ contains an integer if and only if $\frac{2k}{k-j} \in \Z$. Therefore
\begin{align*}
   & \left| \sum_{j=1}^{k-1} (-1)^j {k \choose j} x^{k-j} \sum_{\frac{2k-1}{k-j} < \alpha \leqslant \frac{2k}{k-j}} g_\ell(\alpha) \sum_{p \leqslant x^{1/\alpha}} p^{2k-1 - \alpha (k-j)} \left\lbrace \frac{x}{p^\alpha} \right\rbrace^j \right| \\
   & \leqslant \sum_{j=1}^{k-1} {k \choose j} x^{k-j} \ \sum_{\frac{2k-1}{k-j} < \alpha \leqslant \frac{2k}{k-j}} \left| g_\ell(\alpha) \right| \sum_{p \leqslant x^{1/\alpha}} p^{2k-1 - \alpha (k-j)} \\
   & \leqslant C_0 \sum_{j=1}^{k-1} {k \choose j} \ \sum_{\frac{2k-1}{k-j} < \alpha \leqslant \frac{2k}{k-j}} \alpha^{\ell} x^{\frac{2k-1}{\alpha}} \pi \left( x^{1/\alpha} \right) \\
   & < \frac{2C_0}{\log x} \sum_{j=1}^{k-1} {k \choose j} \ \sum_{\frac{2k-1}{k-j} < \alpha \leqslant \frac{2k}{k-j}} \alpha^{\ell+1} x^{\frac{2k}{\alpha}} \\
   & \leqslant \frac{2C_0(2k)^{\ell+1}}{\log x} \sum_{j=1}^{k-1} {k \choose j} \frac{x^{k-j}}{(k-j)^{\ell+1}} \leqslant \frac{2^{\ell+2} k^{\ell+1} C_0}{\log x} (x+1)^{k-1}
\end{align*}
and hence, for $x \geqslant 2 \times 4^k$
\begin{align*}
   \sum_{\substack{p^\alpha \leqslant x \\ \alpha \geqslant 2}} g_\ell(\alpha) p^{2k-1} \left \lfloor \frac{x}{p^\alpha} \right \rfloor^k &= x^k \sum_{p \leqslant \sqrt{x}} \frac{g_\ell(2) }{p} + x^k \sum_{\substack{p^\alpha \leqslant x \\ \alpha \geqslant 3}} \frac{g_\ell(\alpha)}{p^{1+k(\alpha-2)}}  \\
   & \qquad + \sum_{j=1}^k (-1)^j {k \choose j} x^{k-j} \sum_{2 \leqslant \alpha \leqslant L_2} g_\ell(\alpha)\sum_{p \leqslant x^{1/\alpha}} p^{2k-1 - \alpha (k-j)} \left\lbrace \frac{x}{p^\alpha} \right\rbrace^j \\
   & \qquad \qquad + O_k \left( x^{k-1} (\log x)^{\ell+1} \right)
\end{align*}
where $L_2 := L_2(x,k,j) = \min \left( \frac{2k-1}{k-j} \, , \, L_1 \right)$. We now are in a position to apply Proposition~\ref{pro:frac_parts} with $N$ replaced by $N+\ell+1$, yielding 
\begin{align}
   & \sum_{j=1}^k (-1)^j {k \choose j} x^{k-j} \sum_{2 \leqslant \alpha \leqslant L_2} g_\ell(\alpha) \sum_{p \leqslant x^{1/\alpha}} p^{2k-1 - \alpha (k-j)} \left\lbrace \frac{x}{p^\alpha} \right\rbrace^j \notag \\
   & = \sum_{j=1}^k (-1)^j {k \choose j} x^{k-j} \sum_{2 \leqslant \alpha \leqslant L_2} g_\ell(\alpha) \left( \frac{x^{\frac{2k}{\alpha} - (k-j)}}{\log x} \sum_{h=0}^{N + \ell} \frac{a_{j,h,k,\alpha}}{(\log x)^h} + O \left( \frac{x^{\frac{2k}{\alpha} - (k-j)}}{(\log x)^{N+\ell+2}}\right) \right) \notag \\
   & = \sum_{j=1}^k (-1)^j {k \choose j} \sum_{2 \leqslant \alpha \leqslant L_2} \frac{g_\ell(\alpha) x^{2k/\alpha}}{\log x} \sum_{h=0}^{N + \ell} \frac{a_{j,h,k,\alpha}}{(\log x)^h} \notag \\
   & \qquad + O \left( \sum_{j=1}^k {k \choose j} \sum_{2 \leqslant \alpha \leqslant L_2} \frac{\alpha^{\ell} x^{2k/\alpha}}{(\log x)^{N+\ell+2}} \right) \label{eq:mt+et}
\end{align}
where
$$a_{j,h,k,\alpha} = \int_1^\infty \frac{\{u\}^j (\log u)^h}{u^{1+j-k+2k/\alpha}} \, \textrm{d}u.$$
Recalling that $L_1 = \left \lfloor \frac{\log x}{\log 2} \right \rfloor$, the error term does not exceed
\begin{equation}
   \ll \sum_{j=1}^{k-1} {k \choose j} \sum_{2 \leqslant \alpha \leqslant \frac{2k-1}{k-j}} \frac{\alpha^{\ell} x^{2k/\alpha}}{(\log x)^{N+\ell+2}} + \sum_{2 \leqslant \alpha \leqslant L_1} \frac{\alpha^{\ell} x^{2k/\alpha}}{(\log x)^{N+\ell+2}} \ll_{k,\ell,C_0,N} \frac{x^k}{(\log x)^{N+1}}. \label{eq:et}
\end{equation}
Now let us treat the main term, which is
\begin{align*}
   &= \frac{1}{\log x} \sum_{h=0}^{N + \ell} \frac{1}{(\log x)^h} \sum_{j=1}^k (-1)^j {k \choose j} \sum_{2 \leqslant \alpha \leqslant L_2} g_\ell(\alpha) x^{2k/\alpha} a_{j,h,k,\alpha} \\
   &:= \frac{1}{\log x} \sum_{h=0}^{N + \ell} \frac{1}{(\log x)^h} \times T_{h,k}
\end{align*}
with
\begin{align*}
   T_{h,k} &= \sum_{j=1}^{k-1} (-1)^j {k \choose j} \sum_{2 \leqslant \alpha \leqslant \frac{2k-1}{k-j}} g_\ell(\alpha) x^{2k/\alpha} a_{j,h,k,\alpha} + (-1)^k \sum_{2 \leqslant \alpha \leqslant L_1} g_\ell(\alpha) x^{2k/\alpha} a_{k,h,k,\alpha} \\
   &:= T_1 + T_2
\end{align*}
say. The first sum is
\begin{align*}
   T_1 &= g_\ell(2) x^k \sum_{j=1}^{k-1} (-1)^j {k \choose j} a_{j,h,k,2} + \sum_{j=1}^{k-1} (-1)^j {k \choose j} \sum_{3 \leqslant \alpha \leqslant \frac{2k-1}{k-j}} g_\ell(\alpha) x^{2k/\alpha} a_{j,h,k,\alpha} \\
   &= g_\ell(2) x^k \left( A_{k,h} - (-1)^k a_{k,h,k,2} \right)  + \sum_{j=1}^{k-1} (-1)^j {k \choose j} \sum_{3 \leqslant \alpha \leqslant \frac{2k-1}{k-j}} g_\ell(\alpha) x^{2k/\alpha} a_{j,h,k,\alpha}
\end{align*}
where $A_{k,h}$ is given in \eqref{eq:A_{k,h}}, and using
$$\int_1^\infty \frac{(\log u)^h}{u^{1+\beta}} \, \textrm{d}u = \frac{h!}{\beta^{h+1}} \quad \left( \beta > 0, \ h \in \Z_{\geqslant 0} \right)$$
yields
$$a_{j,h,k,\alpha} \leqslant \frac{h!}{(j-k+2k/\alpha)^{h+1}} \leqslant h! \left( \frac{2k-1}{k-j} \right)^{h+1} \leqslant h! (2k-1)^{h+1}$$
and therefore
$$T_1 = g_\ell(2) x^k \left( A_{k,h} - (-1)^k a_{k,h,k,2} \right)  + O_{h,k,\ell,C_0} \left( x^{2k/3} \right).$$
Similarly, 
$$T_2 = (-1)^k g_\ell(2) x^k a_{k,h,k,2}  + O_{h,k,\ell,C_0} \left( x^{2k/3} (\log x)^{h+\ell+2} \right)$$
so that
$$T_{h,k} = A_{k,h} g_\ell(2) x^k + O_{h,k,\ell,C_0} \left( x^{2k/3} (\log x)^{h+\ell+2} \right)$$
and the main term of \eqref{eq:mt+et} then is 
\begin{equation}
   = \frac{g_\ell(2)x^k}{\log x} \sum_{h=0}^{N + \ell} \frac{A_{k,h}}{(\log x)^h} + O_{k,\ell,C_0,N} \left( x^{2k/3} (\log x)^{\ell+1} \right). \label{eq:mt}
\end{equation}
Now inserting \eqref{eq:mt} and \eqref{eq:et} in \eqref{eq:mt+et} yields
$$\sum_{j=1}^k (-1)^j {k \choose j} x^{k-j} \sum_{2 \leqslant \alpha \leqslant L_2} g_\ell(\alpha) \sum_{p \leqslant x^{1/\alpha}} p^{2k-1 - \alpha (k-j)} \left\lbrace \frac{x}{p^\alpha} \right\rbrace^j = \frac{g_\ell(2) x^k}{\log x} \sum_{h=0}^{N-1} \frac{A_{k,h}}{(\log x)^h} + O_{k,\ell,C_0,N} \left( \frac{x^k}{(\log x)^{N+1}} \right).$$
We finally complete the proof of this case with the estimates
$$\sum_{p \leqslant \sqrt{x}} \frac{1}{p} = \log \log \sqrt{x} + M + O \left( \delta_c \left( \sqrt{x} \right) \right)$$
and, using Lemma~\ref{le:toth} with $x \geqslant e^{\frac{2 \ell+1}{k}}$, we get
\begin{align*}
   x^k \sum_{\substack{p^\alpha \leqslant x \\ \alpha \geqslant 3}} \frac{g_\ell(\alpha)}{p^{1+k(\alpha-2)}} &= x^k \sum_{p \leqslant x^{1/3}} \left( \sum_{\alpha = 3}^\infty \frac{g_\ell(\alpha)}{p^{1+k(\alpha-2)}} + O \left( p^{2k-1} \sum_{\alpha > \frac{\log x}{\log p}}^\infty \frac{\alpha^\ell}{p^{\alpha k}} \right) \right) \\
   &= x^k \sum_p \ \sum_{\alpha = 3}^\infty \frac{g_\ell(\alpha)}{p^{1+k(\alpha-2)}} + O \left( x^k \sum_{p > x^{1/3}} p^{2k-1} \sum_{\alpha = 3}^\infty \frac{\alpha^\ell}{p^{\alpha k}} + \sum_{p \leqslant x^{1/3}} \left( \frac{\log x}{\log p} \right)^\ell p^{2k-1} \right) \\
   &= x^k \sum_p \ \sum_{\alpha = 3}^\infty \frac{g_\ell(\alpha)}{p^{1+k(\alpha-2)}} + O \left( x^k \sum_{p > x^{1/3}} \frac{1}{p^{k+1}} + x^{\frac{2k-1}{3}} (\log x)^\ell \sum_{p \leqslant x^{1/3}} \frac{1}{(\log p)^\ell}\right) \\
   &= x^k \sum_p \ \sum_{\alpha = 3}^\infty \frac{g_\ell(\alpha)}{p^{1+k(\alpha-2)}} + O \left( x^{2k/3}( \log x)^{\ell-1} \right).
\end{align*}
\noindent
Now let us treat the $4$th case $s \geqslant 2k$. For convenience, define
$$\Sigma_{k,\ell,s}:= \sum_{\frac{\sqrt{x}}{\log x} < p \leqslant \sqrt{x}} p^s \sum_{2 \leqslant \alpha \leqslant \frac{\log x}{\log p}} g_\ell(\alpha) \left \lfloor \frac{x}{p^\alpha} \right \rfloor^k.$$
By Lemma~\ref{le:useful_est}, it is sufficient to prove that
$$\Sigma_{k,\ell,s} = \frac{2 g_\ell(2)x^{(s+1)/2}}{(s+1)\log x} \sum_{j=0}^{k-1} (-1)^{k-1-j} {k \choose j} \zeta \left( \tfrac{s+1}{2}-j \right) + O \left\lbrace x^{(s+1)/2} \left( (\log x)^{1-\frac{s+1}{k}} + (\log x)^{-2} \right) \right\rbrace$$
since $s \geqslant 2k \Longrightarrow s+2-2k \geqslant 2$. By Lemmas~\ref{le:transformation_case_2} and~\ref{le:re_alpha_3}, we derive
$$\Sigma_{k,\ell,s} = g_\ell(2) \sum_{n_1 < (\log x)^2} \dotsb \sum_{n_k < (\log x)^2} \ \sum_{\frac{\sqrt{x}}{\log x} < p \leqslant (M_k x)^{1/2}} p^s + O \left( \frac{x^{(s+1)/2}}{(\log x)^{s+2-2k}} \right)$$
and since
$$\sum_{p \leqslant \frac{\sqrt{x}}{\log x}} p^s \ll \left( \frac{\sqrt{x}}{\log x} \right)^{s+1} \frac{1}{\log \left( \sqrt x / \log x \right)} \ll \frac{x^{(s+1)/2}}{(\log x)^{s+2}}$$
we get
$$\left| g_\ell(2) \right| \sum_{n_1 < (\log x)^2} \dotsb \sum_{n_k < (\log x)^2} \ \sum_{p \leqslant \frac{\sqrt{x}}{\log x}} p^s \ll \frac{x^{(s+1)/2}}{(\log x)^{s+2-2k}}$$
and therefore
$$\Sigma_{k,\ell,s} = g_\ell(2) \sum_{n_1 < (\log x)^2} \dotsb \sum_{n_k < (\log x)^2} \ \sum_{p \leqslant (M_k x)^{1/2}} p^s + O \left( \frac{x^{(s+1)/2}}{(\log x)^{s+2-2k}} \right).$$
By Lemma~\ref{le:sum_p_beta} under the weaker form
$$\sum_{p \leqslant z} p^\beta = \frac{z^{\beta+1}}{(\beta+1) \log z} + O \left( \frac{z^{\beta+1}}{(\log z)^2} \right) \quad \left( \beta \geqslant 0 \right),$$
the main term is
\begin{align*}
   & = g_\ell(2)\sum_{n_1 < (\log x)^2} \dotsb \sum_{n_k < (\log x)^2} \left\lbrace \frac{(M_k x)^{(s+1)/2}}{(s+1)\log(\sqrt{M_k x})} + O \left( \frac{(M_k x)^{(s+1)/2}}{(\log(M_k x))^2} \right)  \right\rbrace \\
   & = \frac{2g_\ell(2)x^{(s+1)/2}}{s+1}  \sum_{n_1 < (\log x)^2} \dotsb \sum_{n_k < (\log x)^2} \frac{M_k^{(s+1)/2}}{(s+1)\log(M_k x)} \\
   & \qquad + O \left( \frac{x^{(s+1)/2}}{(\log x)^2} \underbrace{\sum_{n_1 < (\log x)^2} \dotsb \sum_{n_k < (\log x)^2} M_k^{(s+1)/2}}_{\ll 1} \right) \\
   &= \frac{2g_\ell(2)x^{(s+1)/2}}{(s+1) \log x}  \sum_{n_1 < (\log x)^2} \dotsb \sum_{n_k < (\log x)^2} M_k^{(s+1)/2}  \\
   & \qquad + O \left( \frac{x^{(s+1)/2}}{(\log x)^2} + \frac{x^{(s+1)/2}}{(\log x)^2} \underbrace{\sum_{n_1 < (\log x)^2} \dotsb \sum_{n_k < (\log x)^2} M_k^{(s+1)/2} \left| \log M_k \right|}_{\ll 1} \right) 
\end{align*}
and the use of Corollary~\ref{cor:k_qq} and Lemma~\ref{le:remainder} with $m = \frac{1}{2}(s+1)$, along with the bound $g_\ell(2) \ll 2^\ell$, yields
$$\Sigma_{k,\ell,s} = \frac{2g_\ell(2)x^{(s+1)/2}}{(s+1)\log x} \sum_{j=0}^{k-1} (-1)^{k-1-j} {k \choose j} \zeta \left( \tfrac{s+1}{2}-j \right) + O \left( \frac{x^{(s+1)/2}}{(\log x)^2} + \frac{x^{(s+1)/2}}{\log x} \times (\log x)^{2 - \frac{s+1}{k}} +  \frac{x^{(s+1)/2}}{(\log x)^{s+2-2k}} \right)$$
completing the proof since the last error term is absorbed by the sum of the two others.
\end{proof}

\section{Proofs of the main results}

\subsection{Proofs of Theorem~\ref{th:main_1}, Theorem~\ref{th:main_2} and~Corollary~\ref{cor:main_1}}

The proofs are straightforward, using Lemma~\ref{le:additiv_gcd} along with Propositions~\ref{pro:sum_primes_1} and~\ref{pro:sum_primes_2}, and therefore are left to the reader.
\qed

\subsection{Proofs of Theorem~\ref{th:main_3} and Theorem~\ref{th:main_4}}

\begin{proof}[Proof of Theorem~\ref{th:main_3}]
We only treat the case $1 \leqslant r \leqslant k$ and $s < 2r+1$, the other being similar. Split the sum over $j$ in Lemma~\ref{le:key} into three subsums:
\begin{align*}
   \sum_{n_1,\ldots,n_k \leqslant x} f([n_1,\ldots,n_k]) &= \sum_{j=1}^r (-1)^{j-1} {k \choose j} \left \lfloor x \right \rfloor^{k-j} \sum_{n_1,\ldots,n_j \leqslant x} f \left( (n_1,\ldots,n_j) \right) \\
   & \qquad + (-1)^r {k \choose r+1} \left \lfloor x \right \rfloor^{k-r-1} \sum_{n_1,\ldots,n_{r+1} \leqslant x} f \left( (n_1,\ldots,n_{r+1}) \right) \\
   & \qquad \qquad + \sum_{j=r+2}^k (-1)^{j-1} {k \choose j} \left \lfloor x \right \rfloor^{k-j} \sum_{n_1,\ldots,n_j \leqslant x} f \left( (n_1,\ldots,n_j) \right) \\
   & := S_1 + S_2 + S_3,
\end{align*}
say, with the convention that $S_2=S_3=0$ if $k=r$, and $S_3=0$ if $k=r+1$. The main term will be given by $S_1$ for which we may apply Theorem~\ref{th:main_2}, yielding
\begin{align*}
   S_1 &= \sum_{j=1}^r {k \choose j} \left \lfloor x \right \rfloor^{k-j} \Biggl( \frac{\lambda_1 x^{r+1}}{(r+1)\log x} \sum_{h=0}^{j-1} (-1)^{h} {j \choose h} \zeta(r+1-h)  \\
   & \qquad  + O \left(    x^{r+1} \left( (\log x)^{-\frac{r+1}{j}} + (\log x)^{-2} \right) \right) \Biggr) \\
   &= \frac{k \lambda_1 \zeta(r+1) x^{r+1} \left \lfloor x \right \rfloor^{k-1}}{(r+1)\log x} + O \left( \frac{x^{k+r}}{(\log x)^2}\right)  \\
   & \qquad + \sum_{j=2}^r {k \choose j} \left \lfloor x \right \rfloor^{k-j} \Biggl( \frac{\lambda_1 x^{r+1}}{(r+1)\log x} \sum_{h=0}^{j-1} (-1)^{h} {j \choose h} \zeta(r+1-h) \\
   & \qquad \qquad  + O_{k,r} \Bigl( x^{r+k+1}\sum_{j=2}^r x^{-j} \left( (\log x)^{-\frac{r+1}{j}} + (\log x)^{-2} \right) \Bigr) \Biggr).
\end{align*}
Now
\begin{align*}
   \sum_{j=2}^r x^{-j} \left( (\log x)^{-\frac{r+1}{j}} + (\log x)^{-2} \right) &= \left( \sum_{2 \leqslant j \leqslant \frac{r+1}{2}} + \sum_{\frac{r+1}{2} < j \leqslant r} \right) x^{-j} \left( (\log x)^{-\frac{r+1}{j}} + (\log x)^{-2} \right) \\
   & \ll \frac{1}{(\log x)^2} \sum_{2 \leqslant j \leqslant \frac{r+1}{2}} \frac{1}{x^j} + \sum_{\frac{r+1}{2} < j \leqslant r} \frac{1}{x^j (\log x)^{\frac{r+1}{j}}} \\
   & \ll \frac{1}{(\log x)^2} \left( \frac{1}{x^2} + \frac{1}{x^{\frac{r+1}{2}}} \right). 	 
\end{align*}
Using Theorem~\ref{th:main_1}, we get
\begin{align*}
   S_2 &= (-1)^r {k \choose r+1} \left \lfloor x \right \rfloor^{k-r-1} \Biggl( C x^{r+1} \log \log x + G_{r ,s,\ell}(r+1) x^{r+1} \\
   & \qquad + \frac{C x^{r+1}}{\log x} \sum_{h=0}^{N-1} \frac{A_{k,h}}{(\log x)^h}+ O \left( \frac{x^{r+1}}{(\log x)^{N+1}} \right) \Biggr) \ll x^k \log \log x
\end{align*}
and similarly
$$S_3 = \sum_{j=r+2}^k (-1)^{j-1} {k \choose j} \left \lfloor x \right \rfloor^{k-j} \Biggl( F_{r,s,\ell}(j) x^j + O \left( x^{j-1} (\log x)^{\ell+1} \right) \Biggr) \ll x^k$$
completing the proof.
\end{proof}

\begin{proof}[Proof of Theorem~\ref{th:main_4}]
By Lemma~\ref{le:key}, we derive
$$\sum_{n_1,\ldots,n_k \leqslant x} f([n_1,\ldots,n_k]) = k \left \lfloor x \right \rfloor^{k-1} \sum_{n \leqslant x} f(n) + \sum_{j=2}^k (-1)^{j-1} {k \choose j} \left \lfloor x \right \rfloor^{k-j} \sum_{n_1,\ldots,n_j \leqslant x} f \left( (n_1,\ldots,n_j) \right)$$
and using Corollary~\ref{cor:main_1} for the first sum and Theorem~\ref{th:main_1} Form 1 for the second sum yields
\begin{align*}
   & \sum_{n_1,\ldots,n_k \leqslant x} f([n_1,\ldots,n_k]) \\
   &= k \left \lfloor x \right \rfloor^{k-1} \left( C x \log \log x + x G_{0,s,\ell}(1) + \frac{Cx}{\log x} \sum_{h=0}^{N-1} \frac{A_{1,h}}{(\log x)^h} + O \left( \frac{x}{(\log x)^{N+1}}\right) \right) \\
   & \qquad + \sum_{j=2}^k (-1)^{j-1} {k \choose j} \left \lfloor x \right \rfloor^{k-j} \left( x^j F_{0,s,\ell}(j) + O \left( x^{j-1} (\log x)^{\ell+1} \right) \right) \\
   &= k C x^k \log \log x + x^k \left( k G_{0,s,\ell}(1) + \sum_{j=2}^k (-1)^{j-1} {k \choose j} F_{0,s,\ell}(j) \right) \\
   & \qquad + \frac{k C x^k}{\log x} \sum_{h=0}^{N-1} \frac{A_{1,h}}{(\log x)^h} + O \left( \frac{x^k}{(\log x)^{N+1}} + x^{k-1} (\log x)^{\ell+1} \right) 
\end{align*}
implying the asserted result.
\end{proof}

\subsection*{Acknowledgments}

The research of the second author was financed by NKFIH in Hungary, within the framework of the 2020-4.1.1-TKP2020 3rd thematic programme of the University of P\'ecs.

\Addresses

\end{document}